\documentclass[11pt]{amsart}
\usepackage{amsmath}
\usepackage{amssymb}
\usepackage{amsthm}
\usepackage{latexsym}
\usepackage{hyperref}
\usepackage{enumerate}
\usepackage[all]{xy}

\newtheorem{thm}{Theorem}[section]

\newtheorem{lem}[thm]{Lemma}
\newtheorem{prop}[thm]{Proposition}
\newtheorem{defn}[thm]{Definition}
\newtheorem{conj}[thm]{Conjecture}

\newtheorem{example}[thm]{Example}

\newcommand{\enuma}[1]{\begin{enumerate}[\textup{(}a\textup{)}] {#1} \end{enumerate}}

\newcommand{\mr}{\mathrm}
\newcommand{\mc}{\mathcal}
\newcommand{\mf}{\mathfrak}

\newcommand{\Z}{\mathbb Z}

\newcommand{\R}{\mathbb R}
\newcommand{\C}{\mathbb C}

\newcommand{\ochi}{\omega \otimes \chi}

\newcommand{\inp}[2]{\langle #1 \,,\, #2 \rangle}

\newcommand{\matje}[4]{\left(\begin{smallmatrix} #1 & #2 \\ 
#3 & #4 \end{smallmatrix}\right)}

\newcommand{\cent}{C}
\newcommand{\Lpar}{\Psi}

\def\bdd{{\rm bdd}}
\def\cpt{{\rm cpt}}
\def\nr{{\rm nr}}
\def\unr{{\rm unr}}
\def\red{{\rm red}}
\def\rel{{\rm rel}}
\def\semis{{\rm ss}}
\def\Hom{{\rm Hom}}
\def\End{{\rm End}}
\def\Irr{{\rm Irr}}
\def\pro{{\rm pr}}
\def\Gal{{\rm Gal}}
\def\GL{{\rm GL}}
\def\PGL{{\rm PGL}}
\def\Fr{{\rm Frob}}

\def\SL{{\rm SL}}

\def\SO{{\rm SO}}

\def\St{{\rm St}}
\def\triv{{\rm triv}}

\def\cS{{\mathcal S}}
\def\bW{{\mathbf W}}

\begin{document}

\title[Local Langlands correspondence]{On the local Langlands correspondence \\
for non-tempered representations}

\author[A.-M. Aubert]{Anne-Marie Aubert}
\address{Institut de Math\'ematiques de Jussieu -- Paris Rive-Gauche, 
U.M.R. 7586 du C.N.R.S., U.P.M.C., 4 place Jussieu 75005 Paris, France}
\email{aubert@math.jussieu.fr}
\author[P. Baum]{Paul Baum}
\address{Mathematics Department, Pennsylvania State University,  University Park, PA 16802, USA}
\email{baum@math.psu.edu}
\author[R. Plymen]{Roger Plymen}
\address{School of Mathematics, Southampton University, Southampton SO17 1BJ,  England \emph{and} 
School of Mathematics, Manchester University, Manchester M13 9PL, England}
\email{r.j.plymen@soton.ac.uk \quad plymen@manchester.ac.uk}
\author[M. Solleveld]{Maarten Solleveld}
\address{Radboud Universiteit Nijmegen, Heyendaalseweg 135, 6525AJ Nijmegen, the Netherlands}
\email{m.solleveld@science.ru.nl}

\date{\today}

\maketitle

\begin{abstract}
Let $G$ be a reductive $p$-adic group. We study how a local Langlands correspondence for
irreducible tempered $G$-representations can be extended to a local Langlands correspondence
for all irreducible smooth representations of $G$. We prove that, under a natural condition
involving compatibility with unramified twists, this is possible in a canonical way.

To this end we introduce analytic R-groups associated to non-tempered essentially
square-integrable representations of Levi subgroups of $G$. We establish the basic properties
of these new R-groups, which generalize Knapp--Stein R-groups.
\end{abstract}

\tableofcontents

\section*{Introduction}

Let $F$ be a local nonarchimedean field and let $G$ be the group of $F$-rational points
of a connected reductive group which is defined over $F$. Let $\Irr (G)$ be the space of 
irreducible smooth $G$-representations and let 
$\Phi (G)$ be the space of conjugacy classes of Langlands parameters for $G$.
The local Langlands correspondence (LLC) conjectures that there exists an explicit map 
\[
\Irr (G) \to \Phi (G)
\]
which satisfies several naturality properties \cite{Bor}. 
The collection of representations that correspond to a fixed $\phi \in \Phi (G)$ is
known as the L-packet $\Pi_\phi (G)$ and should be finite. A more subtle version 
of the LLC \cite{Vog,Art3}, which for unipotent representations stems from \cite{Lus}, 
asserts that the members of $\Pi_\phi (G)$ can be
parametrized by some irreducible representations $\rho$ of a finite group $\mc{S}_\phi$. 
This leads to a space $\Phi^e (G)$ of enhanced Langlands parameters $(\phi,\rho)$,
and the LLC then should become an injection
\[
\Irr (G) \to \Phi^e (G). 
\]
The proofs of
the LLC for $\GL_n (F)$ \cite{LRS,HaTa,Hen2} are major results. Together with the
Jacquet--Langlands correspondence these provide the LLC for inner forms of $\GL_n (F)$,
see \cite{HiSa,ABPS}. (This has been known for a long time already, but was apparently 
not published earlier.) Recently there has been
considerable progress on the LLC for inner forms of $\SL_n (F)$ \cite{HiSa}
and for quasi-split classical groups \cite{Art4,Mok}. The LLC has been
established for a large class of representations of these groups, including the
collection $\Irr^t (G)$ of irreducible tempered representations.

In general it is expected that is easier to prove the LLC for tempered representations
of a $p$-adic group $G$ than for all irreducible representations. The main reason is 
that every irreducible tempered $G$-representation is unitary and appears as a direct
summand of the parabolic induction of some essentially square-integrable representation.

Therefore a method to generalize the LLC from $\Irr^t (G)$
to $\Irr (G)$ is useful. The aim of this paper is to provide such a method, which is
simple in comparison with the aforementioned papers. The idea is based on the
Langlands classification and to some extent already present in \cite{BrPl,Art3,Sol1}. 
It applies to all reductive groups over local non-archimedean fields. 
Recall that a part of Langlands' conjectures is that $\Irr^t(G)$ corresponds 
to the set $\Phi_{\bdd}(G)$ of bounded Langlands parameters (modulo conjugacy). 
\\[2mm]
\textbf{Theorem \ref{thm:4.2}.} 
\emph{Suppose that a tempered local Langlands correspondence is given as an injective map
$\Irr^t (G) \to \Phi^e_{\bdd} (G)$, which is compatible with twisting by unramified
characters whenever this is well-defined. Then the map extends canonically to a local
Langlands correspondence} $\Irr (G) \to \Phi^e (G)$. \\[1mm]

The main novelty of the paper is the introduction of analytic R-groups for non-tempered
representations (see Definition~\ref{defn:analRgroup} and Theorem~\ref{thm:2.5}). 
These objects, natural generalizations of 
R-groups defined (in the $p$-adic case) by Silberger \cite{Sil1}, open up new ways to compare 
$\Irr^t (G)$ with $\Irr (G)$. Roughly speaking, $\Irr (G)$ is obtained from $\Irr^t (G)$ 
by "complexification" (Proposition \ref{prop:2.6}).

We show that the relation between $\Phi^e_{\bdd}(G)$ and $\Phi^e (G)$ is similar 
(Proposition \ref{prop:1.2}).
As these spaces are not algebraic varieties, a large part of the proof consists of making 
the term "complexification" precise in this context. We do this by constructing suitable 
algebraic families of irreducible representations and of enhanced Langlands parameters.

In Section \ref{sec:geometric} we conjecture how our analytic R-groups are related to
geometric R-groups. This should enable one to produce a LLC for $\Irr (G)$ if the 
Langlands parameters corresponding to essentially square-integrable representations
of Levi subgroups of $G$ are known.

With this in mind we check that the hypotheses of Theorem \ref{thm:4.2} are fulfilled 
in some known cases, in particular for the principal series of a split reductive 
$p$-adic group.

\section{Analytic R-groups for non-tempered representations}
\label{sec:Rgroup}

Let $F$ be a local nonarchimedean field and let $\mathcal G$ be a connected reductive 
algebraic group defined over $F$. We consider the group $G = \mathcal G (F)$ of
$F$-rational points.
Let $P$ be a parabolic subgroup of $G$ with Levi factor $M$, and let $A$ be the maximal
$F$-split torus in the centre of $M$. Then $M = Z_G (A)$ and $N_G (M) = N_G (A)$. 
The Weyl group of $M$ and $A$ is
\[
W(M) = W(A) = N_G (M) / M = N_G (A) / M .
\]
It acts on equivalence classes of $M$-representations by 
\begin{equation}\label{eq:2.1}
(w \cdot \pi)(m) = (\bar w \cdot \pi)(m) = \pi ({\bar w}^{-1} m \bar w), 
\end{equation}
for any representative $\bar w \in N_G (M)$ of $w \in W(M)$.
The isotropy group of $\pi$ is
\[
W_\pi := \{ w \in W(M) : w \cdot \pi \cong \pi \} .
\]
Let $M^1$ be the subgroup of $M$ generated by all compact subgroups of $M$.
Then $M / M^1$ is a lattice and a character of $M$ is unramified if and only if it
factors through $M / M^1$. Let $X_{\nr}(M)$ be the group of unramified characters of $M$ 
and let $X_{\unr}(M)$ be the subgroup of unitary unramified characters. The above 
provides $X_{nr}(M)$ with the structure of a complex torus, such that $X_{\unr}(M)$ is its 
maximal compact subgroup. 

In this paper all representations of $p$-adic groups are tacitly assumed to be smooth.
Let $I_P^G$ be the functor of smooth, normalized parabolic induction, from 
$M$-re\-pre\-sen\-ta\-tions to $G$-representations. The following result is well-known, we
include the proof for a lack of a good reference.

\begin{lem}\label{lem:2.1}
Let $\pi$ be a finite length $M$-representation and take $w \in W(M)$.
Let $P' \subset G$ be another parabolic subgroup with Levi factor $M$. Then the
$G$-representations $I_P^G (\pi) ,\; I_P^G (w \cdot \pi)$ and $I_{P'}^G (\pi)$
have the same trace and the same irreducible constituents, counted with multiplicity.
\end{lem}
\begin{proof}
Conjugation with a representative $\bar w \in N_G (M)$ for $w$ yields an isomorphism 
$I_P^G (w \cdot \pi) \cong I_{w^{-1} P w}^G (\pi)$.
The parabolic subgroup $w^{-1} P w \subset G$ has $M = w^{-1} M w$ as a Levi 
factor, so without loss of generality we may assume that it equals $P'$.

Since $I_P^G (\pi)$ and $I_{P'}^G (\pi)$ have finite length \cite[6.3.8]{Cas}
their irreducible constituents (and multliplicities) are determined by 
their traces \cite[2.3.3]{Cas}. Therefore it suffices to show that the function
\begin{align*}
& C_c^\infty (G) \times X_{\nr}(M) \to \C, \\
& (f,\chi) \mapsto \mathrm{tr}(f,I_P^G (\pi \otimes \chi)) - 
\mathrm{tr}(f, I_{P'}^G (\pi \otimes \chi))
\end{align*}
is identically zero. For a fixed $f \in C_c^\infty (G)$ this is a rational function
on $X_{\nr}(M)$, which by \cite[Th\'eor\`eme IV.1.1]{Wal} vanishes on a
Zariski-dense subset of $X_{\nr}(M)$. Hence it vanishes everywhere.
\end{proof}

Let $X^* (A)$ and $X_* (A)$ be the character (respectively cocharacter) lattice
of $A$. Since $A / (A \cap M^1) \cong X_* (A)$ is of finite index in $M / M^1$,
the restriction map $X_{\nr}(M) \to X_{\nr}(A)$ is surjective and has finite
kernel. In particular there are natural isomorphisms
\[
\{ \chi \in X_{\nr}(M) : \chi (M) \subset \R_{>0} \} \xrightarrow{res}
\Hom_\Z (X_* (A),\R_{>0}) \xrightarrow{log}
X^* (A) \otimes_\Z \R := \mf a^* .  
\]
We note that $\mf a^*$ is a real vector space containing the root system $R(G,A)$.
We say that $\chi \in X_{\nr}(M)$ is positive with respect to $P$ if
\begin{equation}\label{eq:2.10}
\inp{\alpha^\vee}{\log |\chi|} \geq 0 \text{ for all } \alpha \in R (P,A) .
\end{equation}
Let $M(\chi)$ be the maximal Levi subgroup of $G$ such that 
\begin{itemize}
\item $M(\chi) \supset M$ and the split part of $Z(M(\chi))$ is contained in $A$;
\item $\chi$ is unitary on $M \cap M(\chi)_{der}$.
\end{itemize}
Assume that $\pi$ is irreducible and tempered. In particular it is unitary. 
Then $I_{M(\chi) \cap P}^{M(\chi)}(\pi \otimes \chi)$ is completely 
reducible, because its restriction to $M(\chi)_{der}$ is unitary. For every 
irreducible summand $\tau$ of $I_{M(\chi) \cap P}^{M(\chi)}(\pi \otimes \chi)$
the pair $(P M(\chi),\tau)$ satisfies the hypothesis of the Langlands classification
\cite{BoWa,Kon}, so $I_{P M(\chi)}^G (\tau)$ is indecomposable and has a 
unique irreducible quotient $L(P M(\chi),\tau)$. We call the $L(P M(\chi),\tau)$,
for all eligible $\tau$, the Langlands quotients of $I_P^G (\pi \otimes \chi)$.
This subset of $\Irr (G)$ depends only $(M,\pi \otimes \chi)$, because $M(\chi)$
and $P M(\chi)$ are uniquely determined by $\log |\chi|$. We denote it by 
$\Irr_{M,\pi \otimes \chi}(G)$. 

In fact $I_P^G (\pi \otimes \chi)$ is completely reducible for $\chi$ in a 
Zariski-dense subset of $X_{\nr}(M)$. In that case $\Irr_{M,\pi \otimes \chi}(G)$ 
consists of all the consituents of $I_P^G (\pi \otimes \chi)$.

The uniqueness part of the Langlands classification tells us that 
$L(P M(\chi),\tau)$ is tempered if and only if $M(\chi) = G$ and $\tau$ is
tempered. This is so if and only if $\chi$ is unitary, in which case actually
all members of $\Irr_{M,\pi \otimes \chi}(G)$ are tempered.

By Lemma \ref{lem:2.1} the elements of $\Irr_{M,\pi \otimes \chi}(G)$ are also 
constituents of $I_{P'}^G (\pi \otimes \chi)$, so it is justified to call them the 
Langlands constituents of $I_{P'}^G (\pi \otimes \chi)$ for any parabolic subgroup
$P' \subset G$ containing $M$. 

Harish-Chandra showed that every irreducible tempered representation can be obtained 
as a direct summand of the parabolic induction of a square-integrable (modulo centre) 
representation, in an essentially unique way \cite[Proposition III.4.1]{Wal}. 
These considerations lead to the following result.

\begin{thm}\label{thm:2.2} 
\cite[Theorem 2.15]{Sol1} \
\enuma{
\item For every $\pi \in \Irr (G)$ there exist $P,M,\chi$ as above and a
square-integrable (modulo centre) representation $\omega \in \Irr (M)$, such that 
$\pi \in \Irr_{M,\ochi}(G)$.
\item The pair $(M,\ochi)$ is unique up to conjugation.
\item $\pi$ is tempered if and only if $\chi$ is unitary.
} 
\end{thm}

Thus $\Irr (G)$ is partitioned in disjoint packets $\Irr_{M,\ochi}(G)$,
parametrized by conjugacy classes of Levi subgroups $M$ and $W(M)$-equivalence 
classes of essentially square-integrable representations $\ochi \in \Irr (M)$.

We remark that Theorem \ref{thm:2.2} is stronger than the Langlands
classification as formulated in \cite[IV.2]{BoWa} and \cite{Kon}. There the passage 
is from smooth representations to tempered representations, whereas in Theorem 
\ref{thm:2.2} the passage is from smooth representations to essentially 
square-integrable representations (all assumed irreducible of course).
On the other hand, the Langlands classification is one-to-one but
Theorem \ref{thm:2.2} is only finite-to-one.

Let $\omega \in \Irr (M)$ be square-integrable modulo centre and write
\begin{equation}\label{eq:2.12} 
\begin{aligned}
& \mc O = \{ \ochi \in \Irr (M) : \chi \in X_{\unr}(M) \} ,\\
& \mc O_\C = \{ \ochi \in \Irr (M) : \chi \in X_{\nr}(M) \} .
\end{aligned}
\end{equation}
The irreducible constituents of the $G$-representations $I_P^G (\ochi)$
with $\ochi \in \mc O$ make up a Harish-Chandra component
$\Irr_{\mc O} (G)$ of $\Irr^t (G)$, see \cite[\S 1]{SSZ}.
The group 
\[
X_{\nr}(M)_\omega := \{ \chi \in X_{\nr}(M) : \ochi \cong \omega \}
\]
is finite, and in particular consists of unitary characters. The bijection
\begin{equation}\label{eq:2.2}
X_{\nr}(M) / X_{\nr}(M)_\omega \to \mc O_\C : \chi \mapsto \ochi 
\end{equation}
provides $\mc O_\C$ with the structure of a complex torus, and $\mc O$ can be
identified with its maximal real compact subtorus. However, in general there
is no natural multiplication on $\mc O$ or $\mc O_\C$.

Let $W(\mc O)$ be the stabilizer of $\mc O$ in $W(M)$, with respect to the action
\eqref{eq:2.1}. It is also the stabilizer of $\mc O_\C$, and it acts on $\mc O$ 
and $\mc O_\C$ by algebraic automorphisms. 

Recall from \cite[\S 2]{Art1} that for every $w \in W(M)$ and every $\omega \otimes
\chi \in \mc O$ there exists a unitary intertwining operator
\begin{equation}\label{eq:2.3}
J(w,\ochi) \in \Hom_G \big( I_P^G (\ochi),
I_P^G (w \cdot (\ochi)) \big) .
\end{equation}
It is unique up to a complex number of norm 1. If $\chi_G$ is the restriction to $M$
of an unramified character of $G$, then the right hand side of \eqref{eq:2.3} is
$\Hom_G \big( I_P^G (\omega),I_P^G (w \cdot (\omega)) \big)$, so then we may take
\begin{equation}\label{eq:2.5}
J(w,\ochi_G) = J(w,\omega) .
\end{equation}
These operators can be normalized so that $\chi \mapsto J(w,\ochi)$ 
extends to a rational function on $\mc O_\C$. We fix such a normalization. 
It determines a rational function
$\kappa : W(M) \times W(M) \times \mc O_\C \to \C \cup \{\infty\}$ by
\begin{equation}\label{eq:2.4}
J(w,w' \cdot (\ochi)) \circ J(w',\ochi) =
\kappa (w,w',\ochi) J(w w', \ochi) .
\end{equation}
On $\mc O$ this function is regular and takes values of norm 1. By \eqref{eq:2.4}
this holds more generally for all twists of $\omega$ by unramified characters
of $M$ which are unitary on $M \cap G_{der}$. We let
\[
\kappa_{\ochi} : W_{\ochi} \times
W_{\ochi} \to \C \cup \{\infty \}
\] 
be the restriction of $\kappa$ to $W_{\ochi} \times
W_{\ochi} \times \{ \ochi \}$.

\begin{lem}\label{lem:2.3}
$\kappa_{\ochi}$ has neither poles nor zeros.
\end{lem}
\begin{proof}
Let $w \in W_{\ochi}$ and recall the definition of $M(\chi)$, below \eqref{eq:2.10}.
It shows that $w \in N_{M(\chi)}(M)/M$. As we saw above, the operator
\[
J_{M(\chi)}(w,\ochi) \in \End_{M(\chi)}\big( I_{P \cap M(\chi)}^{M(\chi)}
(\ochi) \big)
\]
is regular and invertible. Hence $I_{P M(\chi)}^G \big( J_{M(\chi)}(w,\omega \otimes 
\chi) \big)$ is invertible as well. But all the positive roots of $R(G,A)$ that are 
made negative by $w$ belong to $R(M(\chi),A)$, so 
\[
J(w,\ochi) = z I_{P M(\chi)}^G \big( J_{M(\chi)}(w,\omega \otimes 
\chi) \big) \text{ for some } z \in \C^\times .
\]
Now \eqref{eq:2.4} shows that $\kappa_{\ochi}(w,w') \in \C^\times$
for all $w,w' \in W_{\ochi}$.
\end{proof}

The associativity of the multiplication in $\End_G (I_P^G (\ochi)$ implies 
that $\kappa_{\ochi}$ is a 2-cocycle of $W_{\omega \otimes \chi}$. 
It gives rise to a twisted group algebra $\C [W_{\ochi},\kappa_{\ochi}]$. 
By definition this algebra has a basis $\{J_w : w \in W_{\ochi} \}$ 
and its multiplication is given by
\begin{equation}\label{eq:2.6}
J_w \cdot J_{w'} = \kappa_{\ochi}(w,w') J_{w w'} .
\end{equation}

Let $R_{\red}(G,A)$ be the reduced root system of $(G,A)$. Harish-Chandra's 
$\mu$-function determines a subset
\[
R_{\ochi} := \pm \{ \alpha \in R_{\red}(G,A) : \mu_\alpha (\omega
\otimes \chi) = 0 \} ,
\]
which is known to be a root system itself \cite[\S 1]{Sil1}. Its Weyl group
$W(R_{\ochi})$ is a normal subgroup of $W_{\ochi}$.
The parabolic subgroup $P$ determines a set of positive roots $R^+_{\omega 
\otimes \chi}$. Since $W_{\ochi}$ acts on $R_{\ochi}$,
it is known from the general theory of Weyl groups that the subgroup
\begin{equation}\label{eq:2.13}
\mf R_{\ochi} := \big \{ w \in W_{\ochi} : w (R^+_{\ochi}) = R^+_{\ochi} \big\}
\end{equation}
satisfies 
\begin{equation}\label{eq:2.11}
W_{\ochi} = \mf R_{\ochi} \ltimes W(R_{\ochi}) . 
\end{equation}

\begin{defn} \label{defn:analRgroup}
The group $ \mf R_{\ochi}$ is the analytic R-group attached to the
essentially square-integrable representation $\ochi\in\Irr(M)$.
\end{defn}

\begin{lem}\label{lem:2.4}
Let $Y$ be a connected subset of $\mc O_\C$ such that $W_{\ochi}$
is the same for all $\ochi \in Y$.
\enuma{
\item $R_{\ochi}$ and $\mf R_{\ochi}$ are independent
of $\ochi \in Y$, up to a natural isomorphism.
\item $\C [\mf R_{\ochi},\kappa_{\ochi}]$ and the projective
representation of $W_{\ochi}$ on $I_P^G (\ochi)$ are independent
of $\ochi \in Y$, up to an isomorphism which is determined by the normalization
of the intertwining operators $J_w$.
}
\end{lem}
\begin{proof}
(a) Since $\mu_\alpha$ depends only on the values of the coroot $\alpha^\vee$ on
$\mc O_\C$, it is constant on the connected components of $\mc O_\C^{s_\alpha}$.
Hence $R_{\ochi} = R_{\ochi'}$ for all $\ochi, \ochi' \in Y$. This implies the
corresponding statement for the R-groups, by their very definition.\\
(b) The action of $W_{\ochi}$ via the $J(w,\ochi)$ defines a projective representation
on $I_P^G (\ochi)$. By \cite[\S IV.1]{Wal} the vector space underlying 
$I_P^G (\ochi)$ is independent of $\chi \in X_{\nr}(M) / X_{\nr}(M)_\omega$. 
By Lemma \ref{lem:2.3} the $J(w,\ochi)$ depend algebraically on $\chi$, so we have 
a continuous family of projective representations of the finite group $W_{\ochi}$.
Given the dimension, there are only finitely many equivalence classes of such 
representations, so all the $I_P^G (\ochi)$ with $\ochi \in Y$ are isomorphic as 
projective $W_{\ochi}$-representations. In particular the 2-cocycles $\kappa_{\ochi}$
of $W_{\ochi}$ for different $\ochi \in Y$ are in the same cohomology class. 
Moreover, since the $\kappa_{\ochi}$ are defined in terms of the $J(w,\ochi)$,
they vary continuously as functions on $Y$. Now \eqref{eq:2.6} shows that there
is a unique family algebra isomorphisms
\[
\C [W_{\ochi},\kappa_{\ochi}] \to \C [W_{\ochi'},\kappa_{\ochi'}] 
\quad \text{of the form} \quad J_w \mapsto a_w (\ochi,\ochi') J_w 
\]
with $a_w : Y^2 \to \C^\times$ continuous and $a_w (\ochi,\ochi) = 1$. In view of
part (a) these isomorphisms restrict to
\[
\C [\mf R_{\ochi},\kappa_{\ochi}] \to \C [\mf R_{\ochi'},\kappa_{\ochi'}] . \qedhere
\]
\end{proof}

The following result generalizes the theory of R-groups \cite[\S 2]{Art1} 
to non-tempered representations. It also provides an explanation for the failure of
some properties of R-groups observed in \cite{BaJa}, see Example \ref{ex:SO9}.

\begin{thm}\label{thm:2.5} 
Let $\omega \in \Irr (M)$ be square-integrable modulo centre and let $\chi \in X_{\nr}(M)$. 
\enuma{
\item There exists an injective algebra homomorphism
\[
\C [\mf R_{\ochi},\kappa_{\ochi}] \to \End_G (I_P^G (\ochi)) ,
\]
which is bijective if $\chi$ is positive with respect to $P$. It is canonical up to 
twisting by characters of $\mf R_{\ochi}$.
\item Part (a) determines bijections
\[
\begin{array}{ccccc}
\Irr \big( \C [\mf R_{\ochi},\kappa_{\ochi}] \big) & \to & \Irr_{M,\ochi}(M(\chi)) & \to
& \Irr_{M,\ochi}(G) \\
\rho & \mapsto & \pi (M,\ochi,\rho) & \mapsto & L(M,\ochi,\rho) ,
\end{array}
\]
where 
\[
\pi (M,\ochi,\rho) = \Hom_{\C [\mf R_{\ochi},\kappa_{\ochi}]}
\big( \rho, I_{P \cap M(\chi)}^{M(\chi)} (\ochi) \big)
\] 
and $L(M,\ochi,\rho)$ is the unique Langlands constituent of 
\[
I_{P M(\chi)}^G (\pi (M,\ochi,\rho)) = \Hom_{\C [\mf R_{\ochi},\kappa_{\ochi}]}
\big( \rho, I_P^G (\ochi) \big) .
\]
\item $I_P^G (\ochi) \cong \bigoplus_\rho I_{P M(\chi)}^G (\pi (M,\ochi,\rho)) \otimes \rho$
as $G \times \C [\mf R_{\ochi},\kappa_{\ochi}]$-representations.
} 
\end{thm}
\begin{proof}
For $\chi \in X_{\unr}(M)$ this is well-known, see \cite[\S 2]{Art1}. 
By \eqref{eq:2.5} it holds more generally for $\chi \in X_{\nr}(M)$ which are unitary on
$M \cap G_{der}$.\\
(a) Since $W(\mc O)$ acts on $\mc O_\C$ by algebraic automorphisms, we can find a set $Y$
as in Lemma \ref{lem:2.4} which contains both $\ochi$ and some $\omega' \in \mc O$.
By \cite{Sil1} the intertwining operator $J(w,\omega') \in \End_G (I_P^G (\omega'))$ 
is scalar if and only if $w \in W (R_{\omega'})$, and by the aforementioned result of 
\cite{Art1} the operators $J(w,\omega')$ with $w \in \mf R_{\omega'}$ span a subalgebra of
$\End_G (I_P^G (\omega'))$ isomorphic to $\C [\mf R_{\omega'},\kappa_{\omega'}]$. By Lemma
\ref{lem:2.4}.b the same holds for all elements of $Y$, and in particular for $\ochi$.
By Harish-Chandra's commuting algebra theorem \cite[Theorem 5.5.3.2]{Sil2}
\begin{equation}\label{eq:2.7}
\C [\mf R_{\ochi},\kappa_{\ochi}] \cong \End_G (I_P^G (\omega \otimes \chi)) 
\quad \text{for } \chi \in X_{\unr}(M) .
\end{equation}
Since both sides are invariant under twisting by unramified characters of $G$, \eqref{eq:2.7} 
holds whenever $\chi \in X_{\nr}(M)$ is unitary on $M \cap G_{der}$. 

Every element of $W(G)$ that stabilizes $(M,\ochi)$ already lies in $W(M(\chi))$. Therefore
it does not matter whether we compute $W_{\ochi}$ in $G$ or in $M(\chi)$. The definitions
of $R^+_{\ochi} ,\; W(R_{\ochi})$ and $\mf R_{\ochi}$ are also the same for $(G,P)$ as
for $(M(\chi),P \cap M(\chi))$. Now it follows from \cite[Proposition 2.14.c]{Sol1} and
\eqref{eq:2.7} that for $\chi$ positive with respect to $P$
\[
\End_G (I_P^G (\ochi)) \cong \End_{M(\chi)} \big( I_{P \cap M(\chi)}^{M(\chi)} (\ochi) \big)
\cong \C [\mf R_{\ochi},\kappa_{\ochi}] .
\]
The construction of the isomorphism \eqref{eq:2.7} is unique up to algebra automorphisms
of $\C [\mf R_{\ochi},\kappa_{\ochi}]$ which preserve each of the one-dimensional subspaces
$\C w$. Every such automorphism comes from twisting by a character of $\mf R_{\ochi}$. \\
(c) In view of the remarks at the start of the proof, this holds with respect to the group
$M(\chi)$ (instead of $G$). But $\C [\mf R_{\ochi},\kappa_{\ochi}]$ is the same for
$(G,P)$ and $(M(\chi),P \cap M(\chi))$, so we obtain the result for $G$ by applying the
functor $I_{P M(\chi)}^G$ to the result for $M(\chi)$.\\
(b) For the same reason as (c), this holds on the level of $M(\chi)$. Choose a parabolic 
subgroup $P'$ containing $M$, with respect to which $\chi$ is positive. Then $P' \cap M(\chi)
= P \cap M(\chi)$, so 
\[
\pi (M,\ochi,\rho) = \Hom_{\C [\mf R_{\ochi},\kappa_{\ochi}]}
\big( \rho, I_{P' \cap M(\chi)}^{M(\chi)} (\ochi) \big) .
\]
By Lemma \ref{lem:2.1} $I_P (M,\ochi,\rho)$ and $I_{P'}^G (M,\ochi,\rho)$ have the same
irreducible constituents and by the Langlands classification there is a unique 
Langlands quotient among them. This provides the bijection 
$\Irr_{M,\ochi}(M(\chi)) \to \Irr_{M,\ochi}(G)$. 
\end{proof}

We remark that, since parabolic induction preserves irreducibility of representations 
in most cases, $L (M,\ochi,\rho) = I_{P M(\chi)}^G (\pi (M,\ochi,\rho))$ for $\chi$ 
in a Zariski-open subset of $\mc O_\C$. For $\ochi \in \mc O$ the stronger 
$L (M,\ochi,\rho) = \pi (M,\ochi,\rho)$ holds, because then $M(\chi) = G$.

Theorem \ref{thm:2.5} gives rise to a conjectural parametrization of L-packets.
Suppose that $\phi$ is a Langlands parameter for $G$, which is elliptic for a Levi
subgroup $M \subset G$. By \cite[\S 10.3]{Bor} the L-packet $\Pi_\phi (M)$ should
consist of essentially square-integrable representations. If $N_G (M,\phi)$ denotes the
stabilizer of this L-packet in $G$, Theorem \ref{thm:2.2} shows that $N_G (M,\phi)$-associate 
elements of $\Pi_\phi (M)$ yield the same parabolically induced representations. The
conjectural compatibility of the local Langlands correspondence with parabolic induction
and with the formation of Langlands quotients make it reasonable to expect that
\begin{multline}\label{eq:2.14}
\Pi_\phi (G) = \bigsqcup_{\ochi \in \Pi_\phi (M) / N_G (M,\phi)} \Irr_{M,\ochi}(G) \\
= \bigsqcup_{\ochi \in \Pi_\phi (M) / N_G (M,\phi)} \big\{ L(M,\ochi,\rho) : 
\rho \in \Irr \big( \C [\mf R_{\ochi},\kappa_{\ochi}] \big) \big\} .
\end{multline}

\section{Algebraic families of irreducible representations}
\label{sec:famrep}

Let $X$ be a real or complex algebraic variety. By an algebraic family of 
$G$-re\-pre\-sen\-ta\-tions we mean a family $\{ \pi_x : x \in X\}$ such that all the 
$\pi_x$ are realized on the same vector space (up to some natural isomorphism) and the 
matrix coefficients depend algebraically on $x$. 

Lemma \ref{lem:2.4} and Theorem \ref{thm:2.5} can be used to give a rough description
of the geometric structure of the Bernstein component of $\Irr (G)$ determined by $\mc O$,
in terms of algebraic families. For any subset $Y \subset \mc O_\C$ we define 
\[
\Irr_{M,Y}(G) := \bigcup\nolimits_{\ochi \in Y} \Irr_{M,\ochi}(G) .
\]

\begin{prop}\label{prop:2.6}
Let $Y$ be a maximal connected subset of $\mc O_\C$ on which $W_{\ochi}$ is constant.
\enuma{
\item $Y$ is of the form $X \setminus X^*$, where $X$ is a coset of a complex 
subtorus of $\mc O_\C$ and $X^*$ is a finite union of cosets of complex subtori 
of smaller dimension than $X$. 
\item Let $W(\mc O)_X$ be the (setwise) stabilizer of $X$ in $W(\mc O)$. Theorem 
\ref{thm:2.5} determines a natural bijection 
\[
\big( X \setminus X^* \times \Irr_{M,\ochi}(G) \big) / W(\mc O)_X \to 
\Irr_{M,X \setminus X^*}(G) ,
\]
for any $\ochi \in X \setminus X^*$.
\item Representations in $\Irr_{M,X \setminus X^*}(G)$ are tempered if and only if
the parameter $\ochi$ is in $X_{\cpt} \setminus X^*_{\cpt}$, the canonical real form 
of $X \setminus X^*$.
}
\end{prop}
\begin{proof}
Consider $\mc O_\C$ as an algebraic group via the bijection \eqref{eq:2.2}. The invertible
elements in the coordinate ring $\C [\mc O_\C] \cong \C [X^* (\mc O_\C)]$ are
\[
\C [\mc O_\C ]^\times = \{ z x : x \in X^* (\mc O_\C), z \in \C^\times \} .
\]
Hence the action of $W(\mc O)$ on $\mc O_\C$ induces a group action on
$\C [\mc O_\C ]^\times / \C^\times \cong X^* (\mc O_\C)$, say $(w,x) \mapsto \lambda_w (x)$.
Then
\[
w \cdot zx = z t_w (x) \lambda_w (x) \quad \text{for a unique} \quad t_w (x) \in \C^\times .
\]
Clearly $t_w$ determines a group homomorphism $X^* (\mc O_\C) \to \C^\times$, so it can be
regarded as an element of $\mc O_\C$. Thus we decomposed the action of $w \in W(\mc O)$ on
$\mc O_\C$ as $t_w \lambda_w$, where $\lambda_w$ is an automorphism of $\mc O_\C$ as an
algebraic group and $t_w$ is translation by an element of $\mc O_\C$. The fixed points of
such a transformation are of the form 
\[
\mc O_\C^w = \big( \mc O_\C^{\lambda_w} \big)^\circ F_w
\text{ for some finite subset } F_w \subset \mc O_\C^w .
\]
Furthermore $\big( \mc O_\C^{\lambda_w} \big)^\circ$ is an algebraic torus, since it is the
image of the $\lambda_w$-invariants in the Lie algebra of $\mc O_\C$ under the exponential
map. More generally, for any subgroup $W \subset W(\mc O)$,
\[
\mc O_\C^W = \big( \mc O_\C^{\lambda (W)} \big)^\circ F_W .
\]
The subset $\big( \mc O_\C^W \big)^* \subset \mc O_\C^W$ of points with a stabilizer 
strictly larger than $W$ arises from sets of the same shape, so it is union of cosets 
of algebraic tori of $\big( \mc O_C^{\lambda (W)} \big)^\circ$. For $W = W_Y$ we get
\[
X = \big( \mc O_\C^{\lambda (W_Y)} \big)^\circ (\ochi) \quad \text{and} \quad
X^* = \big( \mc O_\C^{\lambda (W_Y)} \big)^* \cap X ,
\]
which are of the required form.\\
(b) The bijection is constructed with Theorems \ref{thm:2.2}, \ref{thm:2.5} and 
Lemma \ref{lem:2.4}. To see that it is natural, consider a $\ochi \in \mc O \cap X \setminus X^*$
and abbreviate $A = \mathrm{End}_G (I_P^G (\ochi))$. Since all the representations $I_P^G 
(\ochi')$ are realized on the same vector space, Theorem \ref{thm:2.5}.a shows that 
$A \subset \mathrm{End}_G (I_P^G (\ochi'))$ for all $\ochi' \in X \setminus X^*$, and that
$A$ determines the decomposition of $I_P^G (\ochi')$ into indecomposable representations.
If we substitute $A$ for $\C [\mf R_{\ochi},\kappa_{\ochi}]$ in Theorem \ref{thm:2.5}.b we
obtain the same bijection as in part (b) of the current proposition. This makes it clear that
twisting by characters in Theorem \ref{thm:2.5}.a does not effect the bijection, so it is
natural.\\
(c) is merely a restatement of Theorem \ref{thm:2.2}.c.
\end{proof}

For $\rho \in \Irr \big( \C [\mf R_{\ochi},\kappa_{\ochi}] \big)$ we put
\begin{equation}\label{eq:2.8}
\Irr_{M,X \setminus X^*,\rho}(G) = \{ L(M,\ochi,\rho) : \ochi \in X \setminus X^* \} .
\end{equation}
Let $W (\mc O)_{X,\rho}$ be the stabilizer of this set in $W(\mc O)$.
By Proposition \ref{prop:2.6}.b there is a bijection 
\[
(X \setminus X^* ) / W(\mc O)_{X,\rho} \to \Irr_{M,X \setminus X^*,\rho}(G) .
\]
Notice that the left hand side is a complex quasi-affine variety. 
By Proposition \ref{prop:2.6}.c
\begin{multline}\label{eq:2.9}
\Irr^t (G) \cap \Irr_{M,X \setminus X^*,\rho}(G) = \\
\Irr_{M,X_{\cpt} \setminus X_{\cpt}^*,\rho}(G) 
:= \{ \pi (M,\ochi,\rho) : \ochi \in X_{\cpt} \setminus X_{\cpt}^* \} ,
\end{multline}
which is in bijection with the real form $(X_{\cpt} \setminus X^*_{\cpt}) / W(\mc O)_{X,\rho}$ 
of $(X \setminus X^* ) / W(\mc O)_{X,\rho}$. By Theorem \ref{thm:2.2}.b two such families 
$\Irr_{M_1,X_1 \setminus X_1^*,\rho_1}(G)$ and $\Irr_{M_2,X_2 \setminus X_2^*,\rho_2}(G)$ are
either disjoint or equal. The latter happens if and only if there is a $g \in G$ such that
$g M_2 g^{-1} = M_1$ and $(g \cdot X_2 \setminus X_2^*,g \cdot \rho_2)$ is $W(M_1)$-equivalent
with $(X_1 \setminus X_1^*,\rho_1)$. 
In this way $\Irr (G)$ can be regarded as the complexification of $\Irr^t (G)$.

For $\Irr_{M,X \setminus X^*,\rho}(G)$ as in \eqref{eq:2.8}, let $X^\sharp$ be the 
union of $X^*$ and the $\omega \in X \setminus X^*$ for which the Langlands quotient
$L(M,\omega \otimes \chi,\rho)$ is not the whole of $I_{P M(\chi)}^G (\pi (M,\ochi,\rho))$. 
Then $\Irr_{M,X \setminus X^\sharp,\rho}(G)$ is an algebraic family of irreducible
$G$-representations.

\section{Algebraic families of Langlands parameters}
\label{sec:Lpar}

Let $\check G = \check{\mathcal G} (\C)$ be the complex dual group of $G = \mathcal G (F)$.
Let $E/F$ be a finite Galois extension over which $\mathcal G$ splits. The choice of a
pinning (also known as a splitting) for $G$ determines an action of the Galois group
$\Gal (E/F)$ on $\check G$. As Langlands dual group we take
\[
^{L}G = \check G \rtimes \Gal (E/F) . 
\]
Recall that the Weil group of $F$ can be written as 
$\mathbf W_F = \mathbf I_F \rtimes \langle \Fr \rangle$, where $\mathbf I_F$ is the 
inertia subgroup and $\Fr$ is a Frobenius element of $\mathbf W_F$.
A Langlands parameter for $G$ is a continuous group homomorphism 
\[
\phi \colon \mathbf W_F \times \SL_2 (\C) \to {}^{L}G 
\]
such that:
\begin{itemize}
\item $\phi (x) = \phi^\circ (x) \rtimes \pro (x)$, with $\phi^\circ : \mathbf W_F \times 
\SL_2 (\C) \to \check G$ and \\
$\pro:  \mathbf W_F \times \SL_2 (\C) \to \mathbf W_F \to \Gal (E/F)$ the natural projection;
\item $\phi (w)$ is semisimple for $w \in \mathbf W_F$;
\item $\phi \big|_{\SL_2 (\C)} : \SL_2 (\C) \to \check G$ is a homomorphism of algebraic groups.
\end{itemize}
We say that $\phi$ is relevant for $G$ if, whenever the image of $\phi$ is contained in a
parabolic subgroup $^{L}P$ \cite[\S 3]{Bor}, $^{L} P^\circ$ corresponds to a parabolic subgroup
of $\mathcal G$ which is defined over $F$. (This condition is empty if $G$ is quasi-split.)
We define $\Lpar (G)$ to be the set of relevant Langlands parameters for $G$ and 
$\Phi (G)$ to be $\Lpar (G) / \check G$ with respect to the conjugation action.

We say that $\phi \in \Lpar (G)$ is bounded if $\phi (\mathbf W_F)$ is bounded. Since 
$\mathbf I_F$ is compact and $\phi$ is continuous, $\phi$ is bounded if and only if 
$\phi^\circ (\Fr)$ lies in a compact subgroup of $\check G$. We denote the subsets of bounded
elements in $\Lpar (G)$ and $\Phi (G)$ by $\Lpar_{\bdd}(G)$ and $\Phi_{\bdd}(G)$.

\begin{lem}\label{lem:1.1}
Every $\phi \in \Lpar (G)$ can be written as $\phi = \phi_{\nr} \phi_f$ with
$\phi_f \in \Lpar (G) ,\; \phi_f (\mathbf W_F)$ finite and 
\[
\phi_{\nr} : \mathbf W_F \times \SL_2 (\C) / \mathbf I_F \times \SL_2 (\C) \to 
Z_{\check G}(\mathrm{im} \, \phi_f )^\circ .
\]
\end{lem}
\begin{proof}
Heiermann \cite[Lemma 5.1]{Hei} proved the corresponding result for "admissible
homomorphisms" $\mathbf W_F \to {}^{L}G$. His proof remains valid for our Langlands
parameters. Although \cite{Hei} says only that $\phi (\Fr) \in Z_{\check G}(\mathrm{im} \, 
\phi_f )$, the proof shows that $\phi (\Fr)$ lies in the identity component of the latter group.
\end{proof}

We remark that in general $\phi_f$ is not uniquely determined by $\phi$, there can be finitely
many choices for $\phi_f (\Fr)$. 

Suppose now that $\phi_f \in \Lpar (G)$, with $\phi_f (\mathbf W_F)$ finite, is given. For
$s \in Z_{\check G}(\text{im } \phi_f)^\circ$ the element $s \phi_f (\Fr)$ is semisimple if and
only if $s$ is semisimple. In this case there is a Langlands parameter
\[
\phi_{f,s} := \phi_{\nr,s} \phi_f \quad \text{with} \quad \phi_{\nr,s}(\Fr) = s .
\]
Every parabolic subgroup that contains im $\phi_{f,s}$ also contains im $\phi_f$,
so the relevance of $\phi_f$ implies that $\phi_{f,s}$ is relevant for $G$. We put
\[
\Lpar (G,\phi_f) = \{ \phi' \in \Lpar (G) : \phi'_f \sim \phi_f \} , 
\]
where $\sim$ means that $\phi_f$ is a possible choice for $\phi'_f$. Let $\Phi (G,\phi_f)$
be the image of $\Lpar (G,\phi_f)$ in $\Phi (G)$.

Since im$(\phi_f )^\circ = \phi (\SL_2 (\C))$ is reductive, so is 
$Z_{\check G}(\text{im } \phi_f)^\circ$. Lemma \ref{lem:1.1} and 
the above show that $\Lpar (G,\phi_f)$ is naturally parametrized by the set of semisimple elements
$Z_{\check G}(\text{im } \phi_f)^\circ_{\semis}$. Clearly $\Lpar (G)$ is the union (usually not
disjoint) of the subsets $\Lpar (G,\phi_f)$. 
Since $Z_{\check G}(\text{im } \phi_f)^\circ_{\semis}$
is the union of the tori $T$ in $Z_{\check G}(\text{im } \phi_f)^\circ$, we can write
\begin{equation}\label{eq:1.1}
\begin{aligned}
& \Lpar (G) = \bigcup_{\phi_f ,T} \Lpar (G,\phi_f ,T) := 
\bigcup_{\phi_f ,T} \{ \phi_{f,s} : s \in T \} , \\
& \Phi (G) = \bigcup_{\phi_f ,T} \Phi (G,\phi_f ,T) := 
\bigcup_{\phi_f ,T} \big(\text{image of } \Lpar (G,\phi_f ,T) \text{ in } \Phi (G) \big) .
\end{aligned}
\end{equation}
Because all maximal tori of the complex reductive group 
$Z_{\check G}(\mathrm{im} \, \phi_f)^\circ$ are conjugate, we need only one maximal torus $T$ 
for each choice of $\phi_f$ to obtain the whole of $\Phi (G)$. Conjugation by any
element of $\check G$ sends any family $\Lpar (G,\phi_f ,T)$ to another such family, via an 
isomorphism of tori. Consequently
\begin{equation}\label{eq:1.2}
\Phi (G,\phi_{1f},T_1) \cap \Phi (G,\phi_{2f},T_2) \text{ is empty or }
\Phi (G,\phi_{1f},T_1 \cap T'_2) 
\end{equation}
for some torus $T'_2 \subset Z_{\check G}(\mathrm{im} \, \phi_f)^\circ$. We remark that here
and below we allow tori of dimension zero, which are just points.

Let $T_{\cpt}$ denote the maximal compact subgroup of a complex torus $T$, 
so in particular $T$ is the complexification of $T_{\cpt}$. Then
\begin{equation}\label{eq:1.3}
\begin{aligned}
& \Lpar_{\bdd} (G) = \bigcup_{\phi_f ,T} \Lpar (G,\phi_f ,T_{\cpt}) := 
\bigcup_{\phi_f ,T} \{ \phi_{f,s} \in \Lpar (G) : s \in T_{\cpt} \} , \\
& \Phi_{\bdd} (G) = \bigcup_{\phi_f ,T} \Phi_{\bdd} (G,\phi_f ,T) := 
\bigcup_{\phi_f ,T} \big(\text{image of } \Lpar_{\bdd} (G,\phi_f ,T) \text{ in } \Phi (G) \big) .
\end{aligned}
\end{equation}
For $T_2$ and $T'_2$ as in \eqref{eq:1.2}
\begin{equation}\label{eq:1.4}
\Phi \big( G,\phi_{1f},T_{1\cpt} \big) \cap \Phi \big( G,\phi_{2f},T_{2\cpt} \big) 
\text{ is empty or } \Phi \big( G,\phi_{1f},(T_1 \cap T'_2)_{\cpt} \big) .
\end{equation}
By \eqref{eq:1.2} and \eqref{eq:1.4} the intersections between such sets, which are partially
caused by the ambiguity of $\phi \mapsto \phi_f$, do not pose any problems for this way of 
decomposing the space of Langlands parameters. In the sense of \eqref{eq:1.1} and \eqref{eq:1.3}
$\Lpar (G)$ can be regarded as the complexification of $\Lpar_{\bdd}(G)$.  The action of 
$\check G$ preserves the structure introduced above, which enables us to see $\Phi (G)$ as the 
complexification of $\Phi_{\bdd}(G)$.

Now we include the S-groups from \cite{Art3} in the picture. These are improved versions 
of the usual component groups. Let $\check G_{sc}$ 
be the simply connected cover of the derived group of $\check G$. It acts on $\check G$ by
conjugation. For $\phi \in \Lpar (G)$ consider the groups
\[
\cent (\phi) := Z_{\check G_{sc}} (\mathrm{im} \, \phi) \quad \text{and} \quad
\mc{S}_\phi := \cent (\phi) / \cent (\phi)^\circ .
\]
(Arthur calls these groups $S_{\phi,sc}$ and $\widetilde{\mc S}_\phi$.)
Enhanced Langlands parameters for $G$ are pairs 
$(\phi,\rho)$ with $\phi \in \Lpar (G)$ and $\rho \in \Irr (\mc{S}_\phi)$. 
We call the set of such parameters $\Lpar^e (G)$. The conjugation action of $\check G$
on $\Lpar (G)$ extends naturally to an action on $\Lpar^e (G)$, namely
\[
\check g \cdot (\phi,\rho) = (\check g \phi \check{g}^{-1}, \rho \circ \mr{Ad}(\check g)^{-1}) . 
\]
We denote the set of equivalence classes by $\Phi^e (G)$.

Let $X$ be a real or complex algebraic variety. We say that a family 
$\{ (\phi_x ,\rho_x) : x \in X\}$ of enhanced Langlands parameters is an algebraic family 
if $\phi_x \big|_{\mathbf{I}_F \times \SL_2 (\C)}$ is independent of $x, \; \phi_x (\Fr)$ 
depends algebraically on $x$ and all the $\rho_x$ are (in some sense) equivalent.

Let $Z$ be the centralizer in $\check G$ of some element of $\Lpar (G,\phi_f,Y)$, where
$Y$ is a torus as in \eqref{eq:1.1}. Write 
$\mf t_Z = \text{Lie}(Y) \cap Z_{\text{Lie}(\check G)}(Z)$
and put $T_Z = \exp (\mf t_Z)$, a subtorus of $Y$. The elements $\phi \in \Lpar (G,\phi_f,Y)$ with 
$Z_{\check G} (\text{im } \phi) \supset Z$ correspond bijectively to a set of the form
\[
Y_Z = T_Z F_Z \subset Y ,
\]
where $F_Z$ is finite. We remark that $Y_Z$ need not contain the unit element.
The subset of $\phi \in \Lpar (G,\phi_f,Y)$ 
with $Z_{\check G} (\text{im } \phi) \supsetneq Z$ determines a finite union $Y^*_Z$ of cosets of 
algebraic subtori of smaller dimension in $T_Z$. Of course $Y_Z$ can be empty.
Let $T \subset Y_Z$ be a coset of an algebraic
subtorus of $T_Z$ and write $T^* = Y^*_Z \cap T$. For $\rho \in \Irr (\mc{S}_\phi)$ we have
an algebraic family
\begin{equation}\label{eq:1.6}
\Lpar (G,\phi_f,T \setminus T^*,\rho) = \{ (\phi_{f,s},\rho) : s \in T \setminus T^* \} . 
\end{equation}
Let $\Phi (G,\phi_f,T \setminus T^*,\rho)$ be its image in $\Phi^e (G)$. Conjugation by
an element of $\check G$ sends $\Lpar (G,\phi_f,T \setminus T^*,\rho)$ to a family of the
same form. It follows that 
\begin{equation}\label{eq:1.5}
\Phi (G,\phi_{1f},T_1 \setminus T_1^* ,\rho) \cap \Phi (G,\phi_{2f},T_2 \setminus T_2^*,\sigma) 
\text{ is empty or } \Phi (G,\phi_{1f},T_1 \cap T'_2,\rho)  
\end{equation}
for some subtorus $T'_2 \subset Y$. Similarly the set of enhanced bounded Langlands parameters 
$\Lpar^e_{\bdd}(G)$ is a union of the algebraic families
\begin{equation} \label{eq:1.7}
\Lpar (G,\phi_f,T_{\cpt} \setminus T^*_{\cpt},\rho) := 
\{ (\phi_{f,s},\rho) : s \in T_{\cpt} \setminus T^*_{\cpt} \} .
\end{equation}
Again we denote the image in $\Phi^e (G)$ by $\Phi (G,\phi_f,T_{\cpt} \setminus T^*_{\cpt},\rho)$.
By \eqref{eq:1.4} the intersections of such families satisfy
\[
\Phi (G,\phi_{1f},T_{1\cpt},\rho) \cap \Phi (G,\phi_{2f},T_{2\cpt},\sigma) \text{ is empty or }
\Phi (G,\phi_{1f},(T_1 \cap T'_2)_{\cpt},\rho) ,
\]
where $T'_2$ as in \eqref{eq:1.5}. We summarize the findings of this section in a proposition:

\begin{prop}\label{prop:1.2} \
\enuma{
\item $\Lpar^e (G)$ is in a natural way a union of algebraic families 
$\Lpar (G,\phi_f,T \setminus T^*,\rho)$, each of which is parametrized a complex variety 
$T \setminus T^*$. Every $T$ is a coset of a torus in $\check G$, and $T^*$ is a 
(possibly empty) finite union of cosets of tori of smaller dimension.
\item $\Lpar^e_{\bdd}(G)$ is in a natural way a union of algebraic families 
$\Lpar (G,\phi_f,T_{\cpt} \setminus T^*_{\cpt},\rho)$, each of which is parametrized by the 
canonical real form $T_{\cpt} \setminus T^*_{\cpt}$ of the variety $T \setminus T^*$.
\item Via (a) and (b) $\Lpar^e (G)$ can be regarded as the complexification of $\Lpar^e_{\bdd}(G)$.
\item The action of $\check G$ on $\Lpar^e (G)$ preserves these structures, and in that sense
$\Phi^e (G)$ can be seen as the complexification of $\Phi^e_{\bdd}(G)$. 
}
\end{prop}

\begin{example}\label{ex:1.3}
{\rm
We will work out the above families for (enhanced) Langlands parameters for $G = \SL_2 (F)$,
which are trivial on the inertia group $\mathbf I_F$. Put $\check G = \PGL_2 (\C)$ and
let $\check T$ be torus of diagonal elements in $\check G$. The simply connected cover
of $\check G$ is $\check G_{sc} = \SL_2 (\C)$ and we let $\check T_{sc}$ be the torus
of diagonal elements therein. We distinguish the families first by their restriction to
$\SL_2 (\C)$ and then by the possible tori.
\begin{itemize}
\item $\phi \big|_{\SL_2 (\C)} = 1 , \phi_{1f} = 1, T_1 = \check T$.\\
Then $T_1^* = \{ \matje{1}{0}{0}{1} , \matje{1}{0}{0}{-1} \}$ and for all 
$\phi \in \Psi (G,1,T_1 \setminus T_1^*)$ we have $C(\phi) = \check T_{sc}$ and $\mc S_\phi = 1$. 
Moreover
\begin{align*}
& \Phi (G,1,T_1 \setminus T_1^*) = \{ \phi_{\nr,s} : s \in T_1 \setminus T_1^* \} / 
W(\check G,\check T) \cong \big( \C^\times \setminus \{1,-1\} \big) / S_2 , \\
& \Phi (G,1,T_{1\cpt} \setminus T_{1\cpt}^*) \cong 
\{z \in \C^\times : |z| = 1, z \neq 1, z \neq -1 \} / S_2 .
\end{align*} 
\item $\phi \big|_{\SL_2 (\C)} = 1 , \phi_{2f} = 1, T_2 = 1$.\\
Now $T_2^*$ is empty, $C(\phi) = \SL_2 (\C)$ and $\mc S_\phi = 1$. Thus $\Phi (G,1,T_2) = \{1\}$.
\item $\phi \big|_{\SL_2 (\C)} = 1 , \phi_{3f} = \phi_{\nr,\matje{1}{0}{0}{-1}}, T_3 = 1$.\\
In this case $T_3^*$ is empty, $C(\phi) = N_{\check G_{sc}}(\check T_{sc})$ and
$\mc S_\phi = W(\check G_{sc},\check T_{sc}) = S_2$. For every $\rho \in
\Irr (\mc S_\phi)$ we have $\Phi (G,\phi_{3f},T_3,\rho) = \{ (\phi_{3f},\rho)\}$.
\item $\phi \big|_{\SL_2 (\C)}$ the projection $\SL_2 (\C) \to \PGL_2 (\C)$,
$\phi_{4f}$ trivial on $\mathbf W_F$ and $T_4 = 1$.\\
Again $T_4^* = \emptyset$ and there is only one Langlands parameter $\phi = \phi_{4f}$ 
in this family, which satisfies $C(\phi) = Z(\SL_2 (\C)) = \mc S_\phi$. 
For $\rho \in \Irr (\mc S_\phi)$ we obtain $\Phi (G,\phi_{4f},T_4,\rho) = \{ (\phi_{4f},\rho) \}$. 
We remark that for $\rho$ nontrivial $(\phi_{4f},\rho)$ does not parametrize a representation 
of $G$, but one of the essentially unique non-split inner form of $G$.
\end{itemize}
}
\end{example}

\section{From a tempered to a general local Langlands correspondence}
\label{sec:LLC}

In this section we will show how a local Langlands correspondence for $\Irr^t (G)$ can
be extended to $\Irr (G)$. For this purpose we want the enhanced Langlands parameters,
so that every L-packet $\Pi (\phi)$ is split into singletons by $\Irr (\mc S_\phi)$. 
As not all irreducible representations of the S-group $\mc S_\phi$ need to appear
here, we suppose that the LLC is an injective map from $\Irr (G) \to \Phi^e (G)$.
Of course we need to impose additional conditions on this LLC, which we discuss now.

Recall the algebraic families of irreducible $G$-representations and of enhanced Langlands 
parameters from Sections \ref{sec:famrep} and \ref{sec:Lpar}.
We would like to say that via the local Langlands correspondence every algebraic family 
on one side is in bijection with an algebraic family on
the other side. Unfortunately this is not true in general, because our algebraic families 
need not be maximal. In both $\Lpar (G,\phi_f,T \setminus T^*,\rho)$ and 
$\Irr_{M,X \setminus X^*,\rho}(G)$ it is possible that some points of $T^*$ (resp. $X^*$) 
have a larger centralizer in $\check G$ (resp. in $W(\mc O)$), but the same S-groups 
(resp. R-groups) as points of $T \setminus T^*$ (resp. $X \setminus X^*$). Then the 
algebraic family can be extended to a larger subvariety. This behaviour is very common, it 
occurs for most reductive $p$-adic groups. We could 
overcome this problem by adjusting the definitions of $\Lpar (G,\phi_f,T \setminus T^*,\rho)$ 
and $\Irr_{M,X \setminus X^*,\rho}(G)$ so that they include such points of $T$ or $X$. 

However, that would still not imply that our algebraic families are maximal. 
One reason is that $\cent (\phi_{f,t})$ could be larger than $\cent (\phi)$ for $\phi \in \Lpar 
(G,\phi_f,T \setminus T^*)$, but that the subsets of $\Irr (\mc S_{\phi_{f,t}})$ and 
$\Irr (\mc S_\phi)$ that are relevant for the LLC could nevertheless be in natural bijection. 
Even more subtly, it is conceivable that $\mf R_{\ochi}$ is strictly larger than the R-groups 
associated to $M,X \setminus X^*$, but still there exists a 
$\rho_{\ochi} \in \Irr (\mf R_{\ochi},\kappa_{\ochi})$ such that $L(M,\ochi,\rho_{\ochi})$ 
fits in a natural way in $\Irr_{M,X \setminus X^*,\rho}(G)$. Maybe such situations 
could be excluded with more precise conventions and some additional work. 

We prefer to deal with this by proving two versions of our extension theorem: one that 
covers all situations which can theoretically arise inside the framework of the previous
sections, and a more elegant version which works under slightly stronger conditions.

For the first version we assume only that every algebraic family, of the form 
described in Sections \ref{sec:famrep} and \ref{sec:Lpar}, is in correspondence with finitely 
many algebraic families, also as in Sections \ref{sec:famrep} and \ref{sec:Lpar}, on the other 
side (possibly minus some subfamilies of smaller dimension). 

\begin{thm}\label{thm:4.1}
Let a tempered local Langlands correspondence
\[
LL_G^t : \Irr^t (G) \to \Phi^e_{\bdd}(G) 
\]
be given. Suppose that for every algebraic family of irreducible tempered\\
$G$-representations $\Irr_{M,X_{\cpt} \setminus X_{\cpt}^* ,\rho}(G)$ 
as in \eqref{eq:2.8}, there exist
\begin{enumerate}
\item finitely many algebraic families of enhanced bounded Langlands parameters 
$\Lpar (G,\phi_f,T_{i,\cpt} \setminus T_{i,\cpt}^*,\rho_i)$ as in \eqref{eq:1.7};
\item for every $i$, a coset $X_{i,\cpt}$ of a compact subtorus of $X_{\cpt}$ and an 
isomorphism of real algebraic varieties $\psi_i : X_{i,\cpt} \to T_{i,\cpt}$;
\item an injection $\psi : X_{\cpt} \setminus X^*_{\cpt} \to \bigsqcup_i T_{i,\cpt} 
\setminus T^*_{i,\cpt}$;
\end{enumerate}
such that $\psi (\ochi) = \psi_i (\ochi)$ for $\ochi \in \psi_i^{-1}(T_{i,\cpt} \setminus 
T^*_{i,\cpt}) \cap X_{\cpt} \setminus X^*_{\cpt}$, and 
\[
(\phi_{f,\psi (\ochi)},\rho_i) \in \Lpar^e_{\bdd}(G) \quad \text{represents} \quad 
LL_G^t (\pi (M,\ochi,\rho)) .
\]
Then $LL_G^t$ can be extended in a unique way to a map 
\[
LL_G : \Irr (G) \to \Phi^e (G)
\]
such that
\begin{enumerate}
\item the image of $LL_G$ is the complexification of $LL_G^t (\Irr^t (G))$
in the sense of Proposition \ref{prop:1.2};
\item the above conditions hold without the subscripts $\cpt$.
\end{enumerate}
Furthermore $LL_G$ is injective is $LL_G^t$ is so.
\end{thm}
\begin{proof}
By complexification $\psi_i$ extends to an isomorphism of complex algebraic varieties
$\psi_i : X_i \to T_i$. Hence we can extend $\psi$ to an injection
\begin{align*}
& \psi : X \setminus X^* \to \bigsqcup\nolimits_i T_i \setminus T^*_i ,\\ 
& \psi (\ochi) := \psi_i (\ochi) \quad \text{for} \quad
\ochi \in \psi_i^{-1}(T_i \setminus T_i^*) \cap X \setminus X^* .
\end{align*}
Using this we put
\[
LL'_G (M,\ochi,\rho) = (\phi_{f,\psi (\ochi)},\rho_i) \in \Lpar^e (G)
\quad \text{for} \quad \ochi \in \psi^{-1}(T_i \setminus T_i^*) .
\]
Notice that the argument is not a $G$-representation, but a parameter for that.
By assumption $LL'_G (M,\ochi,\rho)$ represents $LL_G^t (\pi(M,\ochi,\rho))$ for  
$\chi \in X_{\unr}(M)$. We want $LL'_G$ to descend to a map 
$\Irr (G) \to \Phi^e (G)$ via Theorem \ref{thm:2.5}.b.
Let us agree to use only one $M$ from every conjugacy class of Levi subgroups of $G$.
In view of Theorem \ref{thm:2.2} it suffices to check that 
\[
LL'_G (L(M,\ochi,\rho)) \text{ is } \check G \text{-conjugate to } 
LL'_G (L(M,w(\ochi),w \rho)) 
\]
for all $w \in W(M)$. By construction this holds if $\chi \in X_{\unr}(M)$. 
Otherwise $|\chi| \in X_{\nr}(M)$ is of infinite order and 
\[
L(M,\ochi \, |\chi |^z,\rho) \cong L(M,w(\ochi \, |\chi |^z),w \rho)
\quad \text{for all } z \in \C . 
\]
For $z \in i \R - 1 ,\; \chi \, |\chi |^z$ is unitary and
\[
LL'_G (M,\ochi \, |\chi |^z,\rho) \text{ is } \check G \text{-conjugate to } 
LL'_G (M,w(\ochi \, |\chi|^z),w \rho) 
\]
because both represent $LL_G^t (\pi (M,\ochi \, |\chi |^z ,\rho))$. These objects vary 
continuously with $z$, so we can find one element $\check g \in \check G$
which conjugates them for all $z \in i \R - 1$ simultaneously. Then $\check g$ actually
works for all $z \in \C$, and in particular for $\ochi$. We conclude that 
$LL'_G$ induces a well-defined map $LL_G : \Irr (G) \to \Phi^e (G)$.

By construction $LL_G$ has all the properties described in the theorem, only the
claim on injectivity is not clear yet. Suppose that 
\[
\phi_1 = LL'_G (M,\omega_1 \otimes \chi_1,\rho_1) \quad \text{and} \quad
\phi_2 = LL'_G (M,\omega_2 \otimes \chi_2,\rho_2)
\]
are conjugate by some element $\check{g}' \in \check G$. Then
\[
|\phi_1 (\Fr)| = |\psi_1 (\omega_1 \otimes \chi_1)| \quad \text {is } 
\check G \text{-conjugate to} \quad |\phi_2 (\Fr)| = |\psi_2 (\omega_2 \otimes \chi_2)| ,
\]
by the same element $\check{g}'$. Hence
\[
LL'_G (M,\omega_1 \otimes \chi_1 |\chi_1 |^z,\rho_1) \quad \text{is } \check G 
\text{-conjugate to} \quad LL'_G (M,\omega_2 \otimes \chi_2 |\chi_2 |^z,\rho_2)
\]
for all $z \in \C$. The injectivity of $LL_G^t$ implies 
\[
L(M,\omega_1 \otimes \chi_1 |\chi_1 |^z,\rho_1) \cong L(M,w(\omega_2 \otimes 
\chi_2 |\chi_2 |^z),\rho_2) \quad \text{for all } z \in i \R - 1.
\]
Proposition \ref{prop:2.6}.b shows that this holds for all $z \in \C$, and
in particular for $z = 0$.
\end{proof}

For a cleaner version of this theorem we summarize the essence of algebraic 
families of irreducible representations in shorter terminology. 
Let $\pi \in \Irr_{M,\omega}(G)$ with $\omega \in \Irr (M)$ square-integrable
modulo centre. For $\chi \in X_{\nr}(M)$ we say that $\pi \otimes \chi$ is
well-defined if there exists a path $t \mapsto \chi_t$ in $X_{\nr}(M)$ with
$\chi_0 = 1$ and $\chi_1 = \chi$, such that there is a canonical isomorphism
$\mf R_{\omega \otimes \chi_t} \cong \mf R_\omega$ for all $t$. This definition
makes sense by Lemma \ref{lem:2.4}, while Proposition \ref{prop:2.6} shows how
$\pi \otimes \chi$ can be constructed. In fact the $\pi \otimes \chi$ which are
well-defined in this sense are precisely the members of a family of representations
as in \eqref{eq:2.8}. Notice also that this convention generalizes the usual
definition of $\pi \otimes \chi$ for $\chi \in X_{\nr}(G)$.

In the above setting there is an inclusion $\check M \to \check G$, unique up
conjugation. We recall a desirable property of the local Langlands correspondence 
from \cite[\S 10]{Bor}: the Langlands parameter of $\pi$ is that of
$\omega$, composed with the map $\check M \to \check G$. Equivalently, it is conjectured
that $\pi$ and $\omega$ have the same Langlands parameter up to conjugation by $\check G$.

The unramified character $\chi$ of $M$ can be regarded as a character of the torus
$M / M_{der}$, whose complex dual group is $Z(\check M)^\circ$. Via the LLC for tori it 
determines a smooth homomorphism 
\[
\phi_\chi : \mathbf{W}_F \to Z(\check M)^\circ \rtimes \mathbf{W}_F ,
\] 
see \cite[\S 8.5 and \S 9]{Bor}. Define $\hat \chi : \mathbf{W}_F \to Z(\check M)^\circ$ 
as the composition of $\phi_\chi$ with the projection on the first coordinate. We extend 
$\hat \chi$ to $\mathbf{W}_F \times \SL_2 (C)$ by making it trivial on $\SL_2 (\C)$.

\begin{thm}\label{thm:4.2}
Suppose that a tempered local Langlands correspondence for $G$ is given as an
injective map
\[
LL_G^t : \Irr^t (G) \to \Phi^e_{\bdd}(G) .
\]
Assume that for all $\pi \in \Irr_{M,\omega}(G)$ we can find
a representative $(\phi_\pi ,\rho_\pi ) \in \Psi^e_{\bdd}(M)$ for $LL_G^t (\pi)$ such that,
whenever $\chi \in X_{\unr}(M)$ and $\pi \otimes \chi$ is well-defined (in the above sense):
\begin{itemize}
\item there is a canonical isomorphism $\alpha_\chi : \mc S_{\phi_\pi} \to 
\mc S_{\phi_{\pi} \hat{\chi}}$;
\item $(\phi_{\pi} \hat{\chi},\alpha_\chi^* \rho_\pi) \in \Psi^e_{\bdd} (M)$ represents 
$LL_G^t (\pi \otimes \chi)$.
\end{itemize} 
Then $LL_G^t$ can be extended in a canonical way to an injection
\[
LL_G : \Irr (G) \to \Phi^e (G)
\]
which fulfills the above conditions for all $\chi \in X_{\nr}(M)$ such that 
$\pi \otimes \chi$ is well-defined.
\end{thm}
\begin{proof}
It suffices to check that the conditions of Theorem \ref{thm:4.1} are fulfilled.
Consider a family $\Irr_{M,X_{\cpt} \setminus X_{\cpt}^* ,\rho}(G)$ and an element
$\omega \in X_{\cpt} \setminus X_{\cpt}^* \subset \Irr^t (M)$. The assumptions enable us
to find a family of Langlands parameters $\Psi (G,\phi_f,T_{\cpt} \setminus T_{\cpt}^*)$
which is in bijection with $\Irr_{M,X_{\cpt} \setminus X_{\cpt}^* ,\rho}(G)$ via
$\pi \otimes \chi \to \phi_\pi \hat{\chi}$. Divide $\Psi (G,\phi_f,T_{\cpt} \setminus 
T_{\cpt}^*)$ into finitely many families of enhanced Langlands parameters 
$\Psi (G,\phi_f,T_{i,\cpt} \setminus T_{i,\cpt}^*,\rho_i)$ according to the different
possibilities for $C(\phi_\pi \hat \chi)$. Here the additional ingredient $\rho_i$ is 
uniquely determined by the second assumption. Now we can apply Theorem \ref{thm:4.1}.
\end{proof}

The conditions of Theorem \ref{thm:4.2} hold in all cases which the authors 
checked, and it seems likely that they are valid for any $p$-adic group $G$ 
(if a tempered local Langlands correspondence exists for $G$). For example 
they hold for all
inner forms of $\GL_n (F)$, because then the component groups and $\mf R$-groups are 
trivial and compatibility with unramified twists is built in the LLC. In fact the usual
LLC for $\GL_n (F)$, denoted $\mathrm{rec}_{F,n}$, fulfills the conditions of Theorem 
\ref{thm:4.2} for non-tempered representations as well. So if we start with $\mathrm{rec}_{F,n} 
\big|_{\Irr^t (\GL_n (F))}$, then Theorem \ref{thm:4.2} yields $\mathrm{rec}_{F,n}$.

The hypotheses are also fulfilled for inner forms of $\SL_n (F)$, as can be deduced from 
\cite{HiSa}. Furthermore both the work of Arthur \cite{Art4} on quasi-split 
orthogonal and symplectic groups and the work on Mok on quasi-split unitary groups
\cite{Mok} should fit with Theorem \ref{thm:4.2}. Indeed, the first
condition in Theorem \ref{thm:4.2} will follow from the comparison of the
analytic and geometric R-groups for tempered representations (see the next
section),  and the second condition should be a consequence of the
functoriality of the twisted endoscopic transfers used in the construction
of the representations of the classical and of the unitary groups.

\section{Geometric R-groups}
\label{sec:geometric}

In the next section we will explain why Theorem \ref{thm:4.2} applies to principal series 
representations of a split reductive $p$-adic group. To that end we first have to improve our 
understanding of the relations between $\mc S_\phi$ and the R-groups from Section 
\ref{sec:Rgroup}, that is, between the analytic and the geometric R-groups. We discuss
this for a general reductive $p$-adic group $G$.

Given $\phi \in \Psi (G)$, let $M$ be a Levi subgroup of $G$ 
such that the image of $\phi$ is contained in ${}^{L}M$, but not in any smaller Levi
subgroup of ${}^{L}G$. We can regard $\mc S_\phi^M$ (that is, $\mc S_\phi$ for $\phi$ 
considered as a Langlands parameter for $M$) as a normal subgroup of $\mc S_\phi$, so the
conjugation action of $\mc S_\phi$ on $\mc S_\phi^M$ induces an action of the quotient
$\mc S_\phi / \mc S_\phi^M$ on $\Irr (\mc S_\phi^M)$. 

\begin{defn} \label{defn:geomRgroup} 
Let $\mf R_{\phi,\sigma}$ be the stabilizer of $\sigma \in \Irr (\mc S_\phi^M)$ in 
$\mc S_\phi / \mc S_\phi^M$. The group $\mf R_{\phi,\sigma}$ is the geometric R-group 
attached to $(\phi,\sigma)$.
\end{defn}

Assume that a local Langlands correspondence for essentially square-integrable 
representations of $M$ is known, and that 
$(\phi,\sigma) \in \Psi^e (M)$ corresponds to $\ochi \in \Irr (M)$.
The following conjecture extends to the 
non-tempered context a conjecture that was stated by Arthur in \cite{Art0}. 
\begin{conj} \label{eq:5.8}
$\mf R_{\phi,\sigma}$ is isomorphic to $\mf R_{\ochi}$.
\end{conj}

Conjecture \ref{eq:5.8} for tempered representations (Arthur's conjecture) 
is known to be true,
in the case when $F$ is of characteristic $0$, 
when $G$ is an inner form of $\SL_n (F)$ (see \cite{ChLi,ChGo}), and 
when $G$ is a classical group, including the case of unitary groups, see
\cite{BaGo} and the references therein. It was also studied, and proven in some 
other cases, in \cite[\S 9]{Ree1}. 

In view of Propositions \ref{prop:2.6} and \ref{prop:1.2}, the validity of Conjecture
\ref{eq:5.8} for all bounded Langlands parameters $\phi \in \Phi_{\bdd}(G)$ would imply
the validity for all $\phi \in \Phi (G)$.

\smallskip

On the other hand, the group $\mf R_{\phi,\sigma}$ for $\phi\in\Psi^e_{\bdd}(G)$  
is a special case of the Arthur group $\mf R_{\psi,\sigma}$ where
$\psi\colon \bW_F\times\SL_2(\C)\times\SL_2(\C)\to {}^L G$ is an Arthur
parameter. Ban and Jantzen \cite{BaJa} provided an example in $G=\SO_9(F)$,
involving an Arthur parameter $\psi$ that has non-trivial restriction to the
second copy of $\SL_2(\C)$, for which the cardinality of $\mf R_{\psi,\sigma}$ 
does not coincide with the number of components of the corresponding
parabolically induced representation. However Conjecture \ref{eq:5.8}
still holds in this case, as we will see. 
   
\begin{example} \label{ex:SO9}
{\rm
Ban and Jantzen considered the representation
\[
\pi = \St_{\GL_2 (F)} \times \triv_{\GL_2 (F)} \rtimes 1,
\] 
which is parabolically induced from
the representation $\St_{\GL_2 (F)} \otimes \triv_{\GL_2 (F)}$ of a Levi subgroup 
of the group $G=\SO_9(F)$.
The representation $\St_{\GL_2 (F)} \otimes \triv_{\GL_2 (F)}$ is not essentially
square-integrable, so we want to compare $\pi$ with the parabolically
induced representation
\[
\sigma = \nu^{1/2}\times \nu^{1/2}\times \St_{\GL_2 (F)} \rtimes 1
\]
where $\nu=|\det|_F$, to which our construction do apply. 
The representation $\pi$ has three constituents \cite[Theorem 2.5]{BaJa}:
\[
\pi=Z(\nu^{-1/2},\nu^{-1/2};\tau_1)+
Z(\nu^{-1/2},\nu^{-1/2};\tau_2)+Z(\nu^{-1/2};\cS),
\] 
where $\tau_1$, $\tau_2$, and $\cS$ are irreducible tempered representations of 
$\SO_5(F)$ defined by
\[
\tau_1+\tau_2=\St_{\GL_2}\rtimes 1\quad\text{and}\quad 
\cS=\St_{\GL_2}\rtimes\St_{\SO_3},
\]
and $Z(\nu^{-1/2},\nu^{-1/2};\tau_i)$ is the unique subrepresentation of the parabolically 
induced representation $\nu^{-1/2}\times \nu^{-1/2}\rtimes\tau_i$, while 
$Z(\nu^{-1/2};\cS)$ is the unique subrepresentation of the parabolically induced
representation $\nu^{-1/2}\rtimes\cS=\nu^{-1/2}\times\St_{\GL_2}\rtimes\St_{\SO_3}$.
We have 
\[
Z(\nu^{-1/2},\nu^{-1/2};\tau_i)\in\Irr_{M,\omega\otimes\chi}(G)
\quad\text{and}\quad Z(\nu^{-1/2};\cS)\in\Irr_{M',\omega'\otimes\chi'}(G),
\]
where $M\simeq F^\times\times F^\times\times \GL_2(F)$ and
$\omega\otimes\chi=\nu^{-1/2}\otimes \nu^{-1/2}\otimes\St_{\GL_2}$, while 
$M'\simeq F^\times\times \GL_2(F)\times\SO_3(F)$ and
$\omega'\otimes\chi'=\nu^{-1/2}\otimes\St_{\GL_2}\otimes\St_{\SO_3}$.
Hence both the $Z(\nu^{-1/2},\nu^{-1/2};\tau_i)$ are Langlands
constituents of $\sigma$ and $Z(\nu^{-1/2};\cS)$ falls into a different series,
because it can be obtained via parabolic induction from a square-integrable
representation of a larger parabolic subgroup.

In this example Arthur R-group $\mf R_{\psi,\sigma}$ has four elements, clearly too many
for the packet. If we do the same calculation as in \cite[\S 2.3]{BaJa} with Langlands 
parameters instead of Arthur parameters, then we end up with a geometric $R$-group of 
order $2$, which is also the analytic $R$-group of $\sigma$. 
Section 3.3 of \cite{BaJa} shows that its irreducible representations naturally
parametrize the first two constituents of $\pi$ discussed above.}
\end{example}

\smallskip 

We write $\Irr (\mc S_\phi, \sigma) = \{ \rho \in \Irr (\mc S_\phi) :
\mathrm{Hom}_{\mc S_\phi^M}(\sigma,\rho) \neq 0 \}$. Choose a minimal idempotent $p_\sigma$ 
of $\C [\mc S_\phi^M]$ associated to $\sigma$. Then the algebra $\C [\mc S_\phi] p_\sigma
\C [\mc S_\phi]$ is Morita equivalent to $p_\sigma \C [\mc S_\phi] p_\sigma$ and
the map $V \mapsto p_\sigma V$ induces a bijection
\begin{equation}\label{eq:5.11}
\Irr (\mc S_\phi, \sigma) = \Irr \big( \C [\mc S_\phi] p_\sigma
\C [\mc S_\phi] \big) \to \Irr \big( p_\sigma \C [\mc S_\phi] p_\sigma \big) .
\end{equation}
On the other hand 
\[
p_\sigma \C [\mc S_\phi] p_\sigma \cong \mathrm{End}_{\mc S_\phi}
\big( \C [\mc S_\phi] p_\sigma \big) \cong  \mathrm{End}_{\mc S_\phi} \big(
\C [\mc S_\phi] \otimes_{\C [\mc S_\phi^M]} \C [\mc S_\phi^M] p_\sigma \big)
= \mathrm{End}_{\mc S_\phi} \big( \mathrm{ind}_{\mc S_\phi^M}^{\mc S_\phi} \sigma \big) .
\]
By \cite[(9.1b)]{Ree1} a choice of intertwining operators 
$I_r \in \mathrm{Hom}_{\mc S_\phi^M} (\sigma,r \cdot \sigma)$ for 
$r \in \mf R_{\phi,\sigma}$ gives rise to a 2-cocycle $\kappa_{\phi,\sigma}$ such that
\begin{equation}\label{eq:5.9}
p_\sigma \C [\mc S_\phi] p_\sigma \cong
\mathrm{End}_{\mc S_\phi} \big( \mathrm{ind}_{\mc S_\phi^M}^{\mc S_\phi} \sigma \big) \cong
\C [\mf R_{\phi,\sigma}, \kappa_{\phi,\sigma}] .
\end{equation}
Thus \eqref{eq:5.11} and \eqref{eq:5.9} provide a bijection between $\Irr (\mc S_\phi,\sigma)$ 
and $\Irr \big( \C [\mf R_{\phi,\sigma}, \kappa_{\phi,\sigma}] \big)$. We remark that in 
general this bijection is not natural, as it can depend on the choice of the intertwining
operators $I_r$.

We call $\sigma \in \Irr (\mc S_\phi^M)$ relevant for $M$ if it corresponds to a representation
of $M$ (as opposed to a representation of an inner form of $M$), and we denote the set of
such $\sigma$ by $\Irr_{\rel M}(\mc S_\phi^M)$. The above action of $\mc S_\phi$ on 
$\Irr (\mc S_\phi^M)$ permutes the different $\mc S_\phi^M$-constituents of a representation
of $\mc S_\phi$, so on the $p$-adic side it should correspond to permuting the different
$\omega \in \Irr (M)$ for which $\Irr_{M,\omega}(G)$ contains a fixed representation of $G$.
Therefore $\Irr_{\rel M}(\mc S_\phi^M)$ should be stable under the action of $\mc S_\phi$.
The desirable properties
of the local Langlands correspondence suggest that there are bijections
\begin{multline}\label{eq:5.3}
\Pi_\phi (G) = \bigsqcup_{\ochi \in \Pi_\phi (M) / N_G (M,\phi)} \Irr_{M,\ochi}(G) 
\longleftrightarrow \\ 
\bigsqcup_{\sigma \in \Irr_{\rel M}(\mc S_\phi^M) / \mc S_\phi} \Irr (\mc S_\phi,\sigma) 
\longleftrightarrow \bigsqcup_{\sigma \in \Irr_{\rel M}(\mc S_\phi^M) / \mc S_\phi} 
\Irr \big( \C [\mf R_{\phi,\sigma}, \kappa_{\phi,\sigma}] \big) .
\end{multline}
Furthermore the comparison with \eqref{eq:2.14} suggests that the cocycles 
$\kappa_{\phi,\sigma}$ and $\kappa_{\ochi}$ should be cohomologous via Conjecture 
\eqref{eq:5.8}. In that case \eqref{eq:2.14} and \eqref{eq:5.3} show how Langlands parameters
for essentially square-integrable representations of Levi subgroups $M$ can be used
to produce a LLC for $G$. This fits well with the work of Heiermann \cite{Hei}, 
who proved that under certain conditions a parametrization of supercuspidal representations
gives rise to one for essentially square-integrable representations.

\section{The principal series of a split group}

From now on we assume that $G$ is $F$-split. The local Langlands correspondence for 
irreducible $G$-representations in the principal series was recently completed in 
\cite{ABPS}. It generalizes \cite{KaLu,Ree2} and relies among others on \cite{Roc}.

First we consider the unramified principal series. Let $(\phi,\sigma) \in \Psi^e_{\bdd}(M)$
be elliptic for a Levi subgroup $M \subset G$ and let $\omega \in \Irr (M)$ be the corresponding
square-integrable (modulo centre) representation. The Kazhdan--Lusztig parametrization of 
irreducible Iwahori-spherical $G$-representations \cite{KaLu,Ree2} is compatible with parabolic 
induction in the sense that this operation does not change the first two ingredients of a 
Kazhdan--Lusztig parameter $(s,u,\rho)$. Since $(s,u)$ determines a Langlands parameter, 
all elements of $\Irr_{M,\omega}(G)$ have Langlands parameter $\phi \in \Phi (G)$.

The appropriate component groups for $G$-representations, at least for the principal series, are
\[
Z_{\check G}(\phi) / Z_{\check G}(\phi)^\circ \cong 
\mc S_\phi / \big( \text{image of } Z(\check G_{sc}) \big) .
\] 
(In other words, the subtlety of replacing $\check G$ by its simply connected cover is 
superfluous for split groups.) We denote the $G$-representation attached to $\phi$ in 
\cite{KaLu} by $\pi_G (\phi)$. By construction 
\[
\pi_G (\phi \hat{\chi}) = \pi_G (\phi) \otimes \chi \text{ for all } \chi \in X_{\nr}(G).
\]
In general $\pi_G (\phi)$ is reducible and endowed with a natural action of 
$Z_{\check G}(\phi) / Z_{\check G}(\phi)^\circ$.
The third ingredient of a Kazhdan--Lusztig parameter is an irreducible representation 
$\rho$ of the latter group. It is used to select an irreducible summand $\pi_G (\phi,\rho)$ of 
$\pi_G (\phi)$, by applying $\mathrm{Hom}_{\mc S_\phi}(\rho,-)$. Choose a 
$\sigma \in \Irr (\mc S_\phi^M)$ which appears in the restriction of $\rho$ to $\mc S_\phi^M$. 
With \eqref{eq:5.11} we obtain
\begin{multline}\label{eq:5.1}
\pi_G (\phi,\rho) \cong \mathrm{Hom}_{\mc S_\phi} (\rho, \pi_G (\phi)) =
\mathrm{Hom}_{\C[\mc S_\phi] p_\sigma \C[\mc S_\phi]} (\rho, \pi_G (\phi)) \\
\cong \mathrm{Hom}_{p_\sigma \C[\mc S_\phi] p_\sigma} (p_\sigma \rho , p_\sigma \pi_G (\phi)) .
\end{multline}
By \cite[Theorem 6.2]{KaLu} $\pi_G (\phi) \cong I_P^G (\pi_M (\phi))$ in an 
$\mc S_\phi^M$-equivariant way. Let $\rho_\sigma \in \Irr \big( \C [\mf R_{\phi,\sigma}, 
\kappa_{\phi,\sigma}] \big)$ correspond to $\rho \in \Irr (\mc S_\phi ,\sigma)$ via \eqref{eq:5.11}
and \eqref{eq:5.9}. It follows that the right hand side of \eqref{eq:5.1} is isomorphic to
\[
\mathrm{Hom}_{\C[\mf R_{\phi,\sigma},\kappa_{\phi,\sigma}]} 
\big( \rho_\sigma , I_P^G (p_\sigma \pi_M (\phi)) \big) .
\]
Because $p_\sigma$ acts as a projection of rank one on the vector space underlying $\sigma$, 
it has essentially the same effect as applying $\mathrm{Hom}_{\mc S_\phi^M}(\sigma,-)$. We find 
\begin{multline}\label{eq:5.2}
\pi_G (\phi,\rho) \cong \mathrm{Hom}_{\C[\mf R_{\phi,\sigma},\kappa_{\phi,\sigma}]} 
\big( \rho_\sigma , I_P^G (\mathrm{Hom}_{\mc S_\phi^M} (\sigma,\pi_M (\phi)) \big) \\
= \mathrm{Hom}_{\C[\mf R_{\phi,\sigma},\kappa_{\phi,\sigma}]} 
\big( \rho_\sigma , I_P^G (\pi_M (\phi,\sigma)) \big) .
\end{multline}
Reeder \cite[\S 9]{Ree1} proved that the analytic R-group $\mf R_{\omega}$ is isomorphic to 
the subquotient $\mf R_{\phi,\sigma}$ of $Z_{\check G}(\phi) / Z_{\check G}(\phi)^\circ$, 
and that the 2-cocycles $\kappa_{\phi,\sigma}$ and $\kappa_{\omega}$ are cohomologous.
From this, \eqref{eq:5.2} and Theorem \ref{thm:2.5}.b it is clear that the way $\Irr \big(
Z_{\check G}(\phi) / Z_{\check G}(\phi)^\circ \big)$ is used here is equivalent to the
method with R-groups in Section \ref{sec:Rgroup}. We already knew that 
$\pi_M (\phi \hat{\chi}) = \pi_M (\phi) \otimes \chi$ for all $\chi \in X_{\nr}(M)$, so
the representations $\pi_G (\phi,\rho) \otimes \chi$ can just as well be constructed via
\eqref{eq:5.2}. Consequently the Kazhdan--Lusztig parametrization satisfies 
$\pi_G (\phi \hat{\chi},\rho) \cong \pi_G (\phi,\rho) \otimes \chi$ whenever this is 
well-defined for a $\chi \in X_{\unr}(M)$. 

Strictly speaking, \cite{KaLu} applies only if $\check G$ has simply connected derived group. 
But Reeder's generalization \cite[Theorem 3.5.4]{Ree2} allows us to forget about this condition, 
as can be seen from equations (93) and (94) of \cite{ABPS}. Thus the assumptions of Theorem 
\ref{thm:4.2} are fulfilled for the unramified principal series of a split group.

In fact the Kazhdan--Lusztig--Reeder parametrization also fulfills the conclusion of Theorem
\ref{thm:4.2}. This can be shown by the above argument, combined with some Langlands quotients
at the appropriate places. The latter do not pose any additional problems, because they
form an integral part of the constructions in \cite{KaLu}.

Next we consider a Bernstein component $[T,\chi]_G$ in the principal series, such 
that the group $\check H = Z_{\check G}(\hat \chi)$ is connected. Since
the local Langlands correspondence for the irreducible representations in $[T,\chi]_G$ 
uses \cite{Roc}, we have to assume that the residual characteristic of $F$ satisfies the
mild conditions in \cite[Remark 4.13]{Roc}. According to \cite[Theorem 9.14]{Roc} the 
block of Rep$(G)$ determined by $[T,\chi]_G$ is equivalent with the block of Rep$(H)$
containing the unramified principal series. This equivalence comes from an isomorphism
of Hecke algebras and it preserves all the important structure, like parabolic 
induction and R-groups. It was checked in \cite[\S 4]{Ree2} that the component groups
$Z_{\check G}(\phi) / Z_{\check G}(\phi)^\circ$ are also preserved in the process.
Since the unramified principal series of $H$ fit in the framework Theorem \ref{thm:4.2},
as shown above, so does the Bernstein component $[T,\chi]_G$.

Finally, suppose that $Z_{\check G}(\hat \chi)$ is disconnected, with identity 
component $\check H$ and component group $\Gamma$. Then everything for the Bernstein 
component $[T,\chi]_G$ can be obtained 
from the setting for $H$, by taking the extended quotient (of the second kind) with
respect to the action of $\Gamma$, see \cite[\S 23]{ABPS}. This procedure is 
essentially the same for enhanced Langlands parameters as for $G$-representations,
and therefore it does not disturb the properties of the local Langlands correspondence
used in Theorem \ref{thm:4.2}.

\end{document}